\documentclass[11pt]{amsart}

\usepackage{amssymb}

\usepackage[a4paper,margin=2.5truecm]{geometry}

\newtheorem{lm}{Lemma}[section]
\newtheorem{thm}[lm]{Theorem}

\newtheorem{prop}[lm]{Proposition}
\newtheorem*{dkp}{DKP Logarithmic Lemma}
\theoremstyle{definition}
\newtheorem{oss}[lm]{Remark}
\newtheorem*{ack}{Acknowledgements}

\numberwithin{equation}{section}

\author{Lorenzo Brasco}
\address{Aix-Marseille Universit\'e, CNRS, Centrale Marseille, I2M, UMR 7373, 13453 Marseille, France}
\email{lorenzo.brasco@univ-amu.fr}
\author{Giovanni Franzina}
\address{Department Mathematik, Universit\"at Erlangen-N\"urnberg Cauerstrasse 11,
91058 Erlangen, Germany}
\email{franzina@math.fau.de}

\title[Convexity properties and Picone-type inequalities]{Convexity properties of Dirichlet integrals and Picone-type inequalities}

\begin{document}

\maketitle

\begin{abstract}
We focus on three different convexity principles for local and nonlocal variational integrals. We prove various generalizations of them, as well as their equivalences.
Some applications to nonlinear eigenvalue problems and Hardy-type inequalities are given. We also prove a measure-theoretic minimum principle for nonlocal and nonlinear positive eigenfunctions.
\end{abstract}


\section{Introduction}
\label{sec:1}

\subsection{A general overview} The aim of this paper is to study three elementary convexity principles which have found many applications in eigenvalue problems
and functional inequalities. In particular, we will focus on their mutual relations and prove that they are indeed equivalent. In order to smoothly introduce the reader to the subject and clarify the scopes of the paper, we start with the three basic examples which will serve as a model for the relevant generalizations considered in the sequel:
\begin{itemize}
\item the first one is the convexity of the Hamiltonian function for a system of one free massive particle in classical mechanics 
\[
\frac{1}{2}\, \frac{|\phi|^2}{m},
\]
which is jointly convex as a function both of the mass $m>0$ and of the conjugate momentum $\phi\in\mathbb{R}^N$;
\vskip.2cm
\item the second one is the convexity of the quantity $|\nabla u|^2$ along curves of the type
\[
\sigma_t=\Big((1-t)\, u^2 + t\, v^2\Big)^\frac{1}{2},\qquad t\in[0,1],
\]
where $u,v\ge 0$ are differentiable functions, {\em i.e.}
\[
|\nabla \sigma_t|^2\le (1-t)\,|\nabla u|^2+t\,|\nabla v|^2.
\]
This is also sometimes called {\it hidden convexity};
\vskip.2cm
\item the third and last one is the so-called\footnote{This formula is called ``identity'' even if it is an inequality, because the difference of the two terms can be written as
\[
\left\langle \nabla u,\nabla\left(\frac{v^2}{u}\right)\right\rangle- |\nabla v|^2
=-\left|\nabla v-\nabla u\, \frac{v}{u}\right|^2,
\]
which is indeed non positive. The latter is the equality which appears in the original
paper \cite{Pi} by Mauro Picone, after whom the formula is named. The identity is used there to obtain comparison principles for ordinary differential equations of Sturm-Liouville type.} {\it Picone identity}
\[
\left\langle \nabla u,\nabla\left(\frac{v^2}{u}\right)\right\rangle\le |\nabla v|^2,
\]
where again $u,v\ge 0$ are differentiable functions, this time with $u>0$.
\end{itemize}
\vskip.2cm
As a well-known consequence of the convexity of the previous Hamiltonian, we get that
\[
\rho\mapsto \frac{|\nabla \rho(x)|^2}{\rho(x)},
\]
is convex for every $x$. This in turn implies convexity of the {\it Fisher information} 
with respect to a reference probability measure\footnote{Here $\mu\ll \nu$ means that $\mu$ is absolutely continuous with respect to $\nu$. We denote by $d\mu/d\nu$ the Radon-Nykodim derivative of $\mu$ with respect to $\nu$.} $\nu$. For every probability measure $\mu$, this functional is given by
\[
\mathcal{J}(\mu\vert\nu) = \int \left|\nabla \log \rho\right|^2\,\rho\, d\nu=\int \frac{|\nabla \rho|^2}{\rho}\, d\nu,\qquad \mbox{ if } \mu\ll \nu\, \mbox{ and }\, \rho=\frac{d\mu}{d\nu},
\]
and observe that the latter can also be re-written as
\[
\mathcal{J}(\mu\vert\nu)= 4\, \int \left|\nabla \sqrt{\rho}\,\right|^2\,d\nu,\qquad \mbox{ if } \mu\ll \nu\, \mbox{ and }\, \rho=\frac{d\mu}{d\nu}.
\]
From the previous we thus get for $\rho_t=(1-t)\, \rho_0+t\,\rho_1$
\[
\int \left|\nabla \sqrt{\rho_t}\,\right|^2\,d\nu\le (1-t)\, \int \left|\nabla \sqrt{\rho_0}\,\right|^2\,d\nu+(1-t)\, \int \left|\nabla \sqrt{\rho_1}\,\right|^2\,d\nu.
\]
Thus in particular if 
\[
\rho_t=(1-t)\, u^2 + t\, v^2,
\]
we then obtain that the Dirichlet integral is convex along curves of the form
\[
\sigma_t=\sqrt{(1-t)\, u^2 + t\, v^2},\qquad t\in[0,1],
\]
which is the {\it hidden convexity} exposed above.
This striking convexity property of the Dirichlet integral seems to have been first noticed by Benguria in his Ph.D. dissertation (see \cite{BBL} for example).
Note that along the curve of functions $\sigma_t$ we have
\[
\|\sigma_t\|^2_{L^2}=(1-t)\, \|u\|_{L^2}^2+t\, \|v\|^2_{L^2},
\]
then if $u,v$ belong to the unit sphere of $L^2$, the same holds true for $\sigma_t$. The latter incidentally happens to be a constant speed geodesic for the metric defined by 
\[
d_2(u,v) = \int \left|u^2-v^2\right|\,dx,\qquad u,v\in L^2. 
\]
As one should expect, the geodesic convexity described above is helpful to get uniqueness results in eigenvalue problems. We recall that eigenvalues of the Dirichlet-Laplace operator $-\Delta$ on an open set $\Omega\subset\mathbb{R}^N$ such that $|\Omega|<\infty$ are defined as the critical points of the Dirichlet integral on the manifold
\[
\mathcal{S}_2(\Omega)=\left\{u\in W^{1,2}_0(\Omega)\, :\, \int_\Omega u^2\, dx=1\right\}.
\]
This constraint naturally introduces Lagrange multipliers, which by homogeneity are the eigenvalues of the
Laplace operator, i.e. any constrained critical point $u$ is a weak solution of
\begin{equation}
\label{helmoltz}
-\Delta u = \lambda\, u, \quad \mbox{ in }\ \Omega,\qquad\qquad u=0, \quad \mbox{ on }\partial\Omega.
\end{equation}
One says that the function $u$ is an {\it eigenfunction} corresponding to the eigenvalue $\lambda$. 
\par
Then the idea is very simple: for a minimum convex problem, critical points are indeed minimizers. This means that hidden convexity trivializes the global analysis for the Dirichlet energy on $\mathcal{S}_2(\Omega)\cap\{u\ge 0\}$ and there cannot be any constant sign critical point $u$ other than its global minimizer.
Since the strong minimum principle states that any constant sign eigenfunction (up to a sign) is strictly positive, this imposes any eigenfunction $v\ge0$ to be associated with the least
eigenvalue
\[
\lambda_1(\Omega) = \min_{u\in W^{1,2}_0(\Omega)} \left\{\int_\Omega |\nabla u|^2\,dx\, :\, \int_\Omega |u|^2\,dx=1\right\},
\]
which turns out to be simple as well.
\par
Another way to prove the same result would be precisely by means of Picone inequality. Let us call $u_1$ a first eigenfunction of $\Omega$, i.e. a function achieving $\lambda_1(\Omega)$. If $v\ge 0$ is a nontrivial eigenfunction with eigenvalue $\lambda$, then by strong minimum principle $v>0$ and by Picone inequality one would get
\[
\lambda=\lambda\,\int_\Omega v\, \frac{u_1^2}{v}=\int_\Omega \left\langle \nabla v,\nabla\left(\frac{u_1^2}{v}\right)\right\rangle\, dx\le \int_\Omega |\nabla u_1|^2\, dx=\lambda_1(\Omega),
\]
and thus $\lambda_1(\Omega)=\lambda$ since $\lambda_1(\Omega)$ is the minimal eigenvalue.
\vskip.2cm
Of course, in the case of the Laplace operator $-\Delta$
simplicity of $\lambda_1(\Omega)$ 
and uniqueness of constant sign eigenfunctions are plain consequences of the Hilbertian structure and of the strong minimum principle for supersolutions
of uniformly elliptic equations.
Indeed, any first eigenfunction $u_1$ must have constant sign and can never vanish on the interior of the connected set $\Omega$. Then any other eigenfunction has to be orthogonal in $L^2(\Omega)$ to $u_1$ (i.e. it has to change sign) unless it is proportional to $u_1$... 

\subsection{Aim of the paper} ...nevertheless, the advantage of the hidden convexity exposed above is that it does not involve any orthogonality concept and applies to general Dirichlet energies of the form
\begin{equation}
\label{energia_H}
\int_\Omega H(\nabla u)\,dx,
\end{equation}
where $z\mapsto H(z)$ is convex, even and positively homogeneous of degree $p>1$. Moreover, we prove in Proposition \ref{lm:hidden} that this remains true for the whole class of interpolating curves
\begin{equation}
\label{curvette}
\sigma_t=\Big((1-t)\, u^q+t\, v^q\Big)^\frac{1}{q},\qquad t\in[0,1].
\end{equation}
with $1\le q\le p$. We point out that $q=1$ corresponds to convexity in the usual sense and that {\it for $q>p$ the property ceases to be true}, see Remark \ref{oss:vital}.
\par
Like in the previous model case $p=q=2$ and $H(z)=|z|^2$, this permits to infer (see Theorem \ref{teo:eigen}) that the only constant sign critical points of \eqref{energia_H}
on the manifold
\[
\mathcal{S}_q(\Omega)=\left\{u\in W^{1,p}_0(\Omega)\, :\, \int_\Omega |u|^q\, dx=1\right\},\qquad 1<q\le p,
\]
are indeed the global minimizers, which are unique up to a sign. Observe that these critical points yield the following nonlinear version of Helmoltz equation \eqref{helmoltz}, i.e.
\[
-\mathrm{div\,}\nabla H(\nabla u) = \lambda\,
 \|u\|_{L^q(\Omega)}^{p-q}\, |u|^{q-2}\,u, \quad \mbox{ in } \Omega,\qquad\qquad u=0\quad \mbox{ on } \partial\Omega.
\]
We refer the reader to \cite{FL} for a detailed account on this nonlinear eigenvalue problem in the case $H(z)=|z|^p$.
\vskip.2cm
The previous general version of the hidden convexity can be seen again as a consequence of the joint convexity of the generalized Hamiltonian\footnote{The parameters $\beta$ and $q$ are linked through the relation
\[
\frac{\beta}{p}+\frac{1}{q}=1.
\]}
\begin{equation}
\label{HamC}
(m,\phi)\mapsto \frac{H(\phi)}{m^\beta},  \qquad \mbox{ for }\ 0\le \beta\le p-1,
\end{equation}
which in turn gives the convexity of the information functional
\[
\begin{split}
\mathcal{J}_{H,\beta}(\mu\vert\nu) = \int H\left(\nabla \log \rho\right)\,\rho^{p-\beta}\,d\nu, \qquad \mbox{ if } \mu\ll \nu \mbox{ and }\rho = \frac{d\mu}{d\nu},
\end{split}
\]
where the latter can also be written as
\[
\mathcal{J}_{H,\beta}(\mu|\nu)=\left(\frac{p}{p-\beta}\right)^p\,\int H\left(\nabla \rho^\frac{p-\beta}{p}\right)\,d\nu, \qquad \mbox{ if } \mu\ll \nu \mbox{ and }\rho = \frac{d\mu}{d\nu}.
\]
Finally, hidden convexity is in turn {\it equivalent} (see Section \ref{sec:3}) to the validity of the following generalized Picone inequality
\begin{equation}
\label{picone}
\left\langle \nabla H(\nabla u), \nabla\left( \frac{v^{q}}{u^{q-1}}\right) \right\rangle \le H(\nabla v)^\frac{q}{p}\, H(\nabla u)^\frac{q-p}{p},
\end{equation}
for all differentiable functions $u,v\ge 0$ with $u>0$, which is proved in Proposition \ref{lm:pitone}. Here again we consider $1< q\le p$. 
\par
We point out that equivalence between Picone-type inequalities and the hidden convexity property seemed to be unknown: indeed, one of the main scopes of this paper is to precise the relation between these two properties. 
\vskip.2cm
Up to now, we have discussed applications of these convexity principles to uniqueness issues in linear and nonlinear eigenvalue problems. But Picone-type inequalities can be used to prove a variety of different results. Without any attempt of completeness (we refer to the seminal paper of Allegretto and Huang \cite{AH} and to the recent paper \cite{Ja} for a significant account on the topic), we focus on applications to Hardy-type functional inequalities.
\par
The idea is that when $u$ solves a quasilinear equation with principal part given by
\[
-\mathrm{div} \nabla H(\nabla u),
\] 
by integrating \eqref{picone} and using the equation one can get a lower bound on $\int H(\nabla v)$ which does not depend on derivatives of $u$. This procedure is now well understood, see the recent paper \cite{DP}. 
\par
In Theorem \ref{thm:hardygen} this is applied to get a sharp anisotropic version of the Hardy inequality which reads as follows
\begin{equation}
\left(\frac{N+\gamma-p}{p}\right)^p\,\int_{\mathbb{R}^N} |v|^p\, F_*(x)^{\gamma-p}\,\, dx\le \int_{\mathbb{R}^{N}} F(\nabla v)^p\, F_*(x)^\gamma\, dx,\quad v\in C^\infty_0(\mathbb{R}^N\setminus\{0\}),
\end{equation}
for $1<p<N$ and $\gamma>p-N$. Here $F$ is any $C^1$ strictly convex norm
and $F_\ast$ denotes the corresponding dual norm, see Section~\ref{ssec:hardyloc}. For the case $\gamma=0$ a different proof, based on
symmetrization arguments, can be found in \cite{Van}.
\par
The same method can be used to get, for example, the following nonlocal version of the Hardy inequality
\begin{equation}
\label{hardyfract}
C \int_{\mathbb{R}^N} \frac{|v|^p}{ |x|^{s\,p}}\,dx \le \int_{\mathbb{R}^N}\int_{\mathbb{R}^N} \frac{\big| v(x)-v(y)\big|^p}{|x-y|^{N+s\,p}}\,dx\,dy\,,\qquad
\mbox{for all} \,\, v\in W^{s,p}_0(\mathbb{R}^N)\setminus\{0\}\,,
\end{equation}
by means of the following discrete version of Picone inequality
\begin{equation}
\label{piconefrac}
|u(x)-u(y)|^{p-2}\, (u(x)-u(y))\,
\left[\frac{v(x)^p}{u(x)^{p-1}}-\frac{v(y)^p}{u(y)^{p-1}}\right]
\le|v(x)-v(y)|^p,
\end{equation} 
with the choice $u(x) = |x|^{-s\,p}$. The constant $C=C(N,s,p)>0$ is sharp and for the sake of completeness we provide details about its computation in Appendix \ref{sec:B}.
We point out that \eqref{hardyfract} was proved by this same method by Frank and Seiringer in \cite{FS}, which is there called {\it ground state substitution}. 
Other fractional Hardy inequalities have appeared in the literature, see~\cite{BD,DF}. In particular, in the recent paper \cite{DdPW} Davila, del Pino and Wei observed that
a suitable fractional Hardy inequality on surfaces plays a role in the stability of nonlocal minimal
cones, see \cite[Corollary 11.1]{DdPW}.
\vskip.2cm
Noteworthy, not only does the Picone
inequality \eqref{picone} have its discrete counterpart \eqref{piconefrac}, but
also the hidden convexity of the Dirichlet integral has a nonlocal version.
Indeed, the Gagliardo seminorm
\[
\int\int \frac{|u(x)-u(y)|^p}{|x-y|^{N+s\,p}}\,dx\,dy
\]
turns out to be convex along curves of the type \eqref{curvette}, whenever $u,v$ are positive. Correspondingly, fractional Picone inequalities (or equivalently
hidden convexity) are used to get uniqueness results for positive eigenfunctions of the
integro-differential operator defined by the following principal value integral 
\[
(-\Delta_p)^s u(x) =2\ \mathrm{ p.v.} \int_{\mathbb{R}^N} \frac{|u(x)-u(y)|^{p-2} (u(x)-u(y))}{|x-y|^{N+s\,p}}\,dy.
\]
We point out that in order to get uniqueness results for constant sign nonlocal eigenfunctions, i.e. for solutions of
\[
(-\Delta_p)^s u=\lambda\, u^{q-1},\quad \mbox{ in }\Omega,\qquad\qquad u=0,\quad \mbox{ in }\mathbb{R}^N\setminus\Omega,
\] 
as in the local case, one needs to know that non-negative eigenfunctions are indeed strictly positive, at least for $\Omega$ connected.
A proof of this {\it strong minimum principle} for nonlocal eigenfunctions is provided in the appendix and seems to be new. The proof is based on a logarithmic lemma recently established in \cite{DKP}. 
After the acceptance of the present paper, we were informed about the interesting manuscript~\cite{DKP2} which contains the proof of a
weak Harnack's inequality for supersolutions (as well as a proper Harnack's inequality for solutions) of the operator $(-\Delta_p)^s$.
The interested reader may find in that article a more detailed account about nonlocal Harnack's inequalities. Nevertheless, those results are not used in this paper.

\subsection{Plan of the paper}
In Section \ref{sec:2} we present and prove some generalizations of the three convexity principles discussed above, then Section \ref{sec:3} is devoted to discuss their equivalences. Section \ref{sec:4} deals with the nonlocal/discrete versions of these convexities. Applications are then given in Sections \ref{sec:5} and \ref{sec:6}. The paper is concluded by two Appendices: one contains a new strong minimum principle for positive nonlocal eigenfunctions (Theorem \ref{teo:minprin}), while the second contains some computations related to the determination of the sharp constant in \eqref{hardyfract}.
\begin{ack}
We thank Agnese Di Castro, Tuomo Kuusi and Giampiero Palatucci for having kindly provided us a copy of their work \cite{DKP}, as well as Enea Parini for pointing out a flaw in a preliminary version of the proof of Theorem A.1. Part of this work has been done during the conferences ``{\it Linear and Nonlinear Hyperbolic Equations}'' and ``{\it Workshop on Partial Differential Equations and Applications}'', both held in Pisa and hosted by Centro De Giorgi and the Departement of Mathematics of the University of Pisa. We acknowledge the two institutions as well as the organizers for the nice atmosphere and the excellent working environment.
The second author has been supported by the ERC Starting Grant No. 258685 ``AnOptSetCon''.
\end{ack}

\section{Three convexity principles}
\label{sec:2}

We start with the a couple of classical results, which will be useful in order to prove some of the results of the paper. We give a proof for the reader's convenience.
\begin{lm}
\label{lm:convexity}
Let $F:\mathbb{R}^N\to\mathbb[0,+\infty)$ be a positively $1-$homogeneous function, i.e.
\[
F(\lambda\, z)=\lambda\, F(z),\qquad z\in\mathbb{R}^N,\ \lambda\ge 0,
\]
which is level-convex, i.e.
\begin{equation}
\label{levelconvex}
F((1-t)\,z+t\, w)\le \max\{F(z),F(w)\},\qquad z,w\in\mathbb{R}^N,\ t\in[0,1].
\end{equation}
Then $F$ is convex.
\end{lm}
\begin{proof}
Let $x_0,x_1\in\mathbb{R}^N$, if $F(x_0)=F(x_1)=0$ then by \eqref{levelconvex}
\[
F((1-t)\,x_0+t\,x_1)=0,\qquad t\in[0,1].
\] 
Let us now suppose for example that $F(x_0)>0$
and take $\varepsilon>0$, we define
\[
z=\frac{x_0}{F(x_0)},\quad w=\frac{x_1}{F(x_1)+\varepsilon}\quad \mbox{ and }\quad t=\frac{F(x_1)+\varepsilon}{F(x_0)+F(x_1)+\varepsilon}.
\]
By using the $1-$homogeneity of $F$, we then obtain
\[
\begin{split}
F((1-t)\, z+t\, w)&=\frac{F(x_0+x_1)}{F(x_0)+F(x_1)+\varepsilon},
\end{split}
\]
while
\[
\max\{F(z),F(w)\}=\max\left\{1,\frac{F(x_1)}{F(x_1)+\varepsilon}\right\}=1.
\]
Then \eqref{levelconvex} implies
\[
F(x_0+x_1)\le F(x_0)+F(x_1)+\varepsilon,\qquad x_0,x_1\in\mathbb{R}^N,
\]
and since $\varepsilon>0$ is arbitrary, we get
\[
F(x_0+x_1)\le F(x_0)+F(x_1),\qquad x_0,x_1\in\mathbb{R}^N,
\]
i.e. $F$ is subadditive.
This in turn implies the desired result, as
\[
F((1-t)\, x_0+t\, x_1)\le F((1-t)\, x_0)+F(t\, x_1)=(1-t)\, F(x_0)+t\, F(x_1),
\]
which concludes the proof.
\end{proof}
\begin{lm}
\label{lm:analisi0}
Let $1<p<\infty$ and let $H:\mathbb{R}^N\to[0,+\infty)$ be a $C^1$ positively $p-$homogeneous convex function. If $H(z)=0$ then we have $\nabla H(z)=0$ as well.
\end{lm}
\begin{proof}
The statement is evident if $z=0$, thus let us suppose that $z\not=0$. Assume on the contrary that $\nabla H(z)\not= 0$, then there exists $h\in\mathbb{R}^N$ with unit norm such that $\langle\nabla H(z),h\rangle=|\nabla H(z)|$. The function $g(t)=H(z+t\,h)$ has the following properties
\[
g\in C^1(\mathbb{R}),\qquad g(t)\ge 0,\qquad g'(0)=|\nabla H(z)|>0=g(0).
\]
This gives a contradiction, thus $\nabla H(z)=0$.
\end{proof}

\subsection{Convexity of generalized kinetic energies}

The first convexity principle we consider is the following.
\begin{prop}
\label{lm:convessa}
Let $1<p<\infty$ and let $H:\mathbb{R}^N\to[0+\infty)$ be a convex positively $p-$homogeneous function. For every $0<\beta\le p-1$ the function 
\[
(m,\phi)\mapsto \frac{H(\phi)}{m^{\beta}},\qquad (m,\phi)\in(0,+\infty)\times\mathbb{R}^N,
\]
is convex. 
\end{prop}
\begin{proof}
For $\beta=p-1$, it is sufficient to observe that
\begin{equation}
\label{fenchel}
\frac{H(\phi)}{m^{p-1}}=\sup_{(t,\xi)} \left\{t\,m+\langle\xi,\phi\rangle\, :\, t+H^*(\xi)\le 0\right\},\qquad m>0,\ \phi\in\mathbb{R}^N,
\end{equation}
where $H^*$ denotes the Legendre-Fenchel transform of $H$. This would give the desired result, since the supremum of affine functions is a convex function. 
\par
For completeness, we verify formula \eqref{fenchel}: since for every $m>0$ the map $t\mapsto t\,m$ is increasing, the maximization in \eqref{fenchel} is unchanged if we replace the inequality constraint by the condition $t+H^*(\xi)=0$. Then the right-hand side of \eqref{fenchel} is equivalent to
\[
\sup_{\xi\in\mathbb{R}^N} \langle\xi,\phi\rangle-H^*(\xi)\, m=m\, \left[\sup_{\xi\in\mathbb{R}^N}\, \left\langle\xi,\frac{\phi}{m}\right\rangle-H^*(\xi)\right]=m\, H^{**}\left(\frac{\phi}{m}\right),
\]
which gives the desired conclusion, by using that $H^{**}=H$ and the positive homogeneity of $H$.
\vskip.2cm\noindent
For $0<\beta<p-1$, let us set for simplicity
\[
\Phi(m,\phi)=\frac{H(\phi)}{m^{p-1}}\qquad \mbox{ and }\qquad \psi(m)=m^\vartheta,\qquad m>0,\ \phi\in\mathbb{R}^N,
\]
where $\vartheta=\beta/(p-1)<1$, then we can rewrite 
\[
\frac{H(\phi)}{m^\beta}=\Phi(\psi(m),\phi),
\]
where $\Phi$ is jointly convex thanks to the first part of the proof and decreasing in its first argument, while $\psi$ is concave, then it is standard to see that their composition is a convex function. Indeed, for every $t\in[0,1]$, $m_0,m_1>0$ and $\phi_0,\phi_1\in\mathbb{R}^N$, we get
\[
\begin{split}
\Phi\Big(\psi((1-t)\, m_0+t\, m_1),(1-t)\, \phi_0+t\, \phi_1\Big)&\le \Phi\Big((1-t)\, \psi(m_0)+t\, \psi(m_1),(1-t)\, \phi_0+t\, \phi_1\Big)\\
&\le (1-t)\, \Phi(m_0,\phi_0)+t\, \Phi(m_1,\phi_1),\\
\end{split}
\]
which gives the desired result.
\end{proof}
A couple of comments on the previous result are in order.
\begin{oss}
As already recalled in the Introduction, the main instance of functions considered in Proposition \ref{lm:convessa} is the following one
\begin{equation}
\label{brenier}
(m,\phi)\,\mapsto \frac{|\phi|^2}{m},\qquad \phi\in\mathbb{R}^N, m>0.
\end{equation}
If one regards the scalar quantity $m$ as a {\it mass} and the vector quantity $\phi$ as the {\it moment} of this mass, i.e. if we decompose $\phi$ as $\phi=v\, m$ with $v\in\mathbb{R}^N$ (the velocity of the mass particle), then we would have
\[
\frac{|\phi|^2}{m}=|v|^2\, m.
\]
This simple remark is the crucial ingredient of the so-called {\it Benamou-Brenier formula} for the $2-$Wasserstein distance (see \cite{BB,Br}). The latter is a distance on the space of probability measures $\mathcal{P}(\Omega)$ over $\Omega$, defined by
\[
w_2(\rho_0,\rho_1)^2:=\inf_T\left\{\int_{\Omega\times\Omega} |x-T(x)|^2\, d\rho_0\, :\, T_\#\rho_0=\rho_1\right\},
\]
for every $\rho_0,\rho_1\in\mathcal{P}(\Omega)$. Here $T_\# \rho_0$ denotes the push-forward of the measure $\rho_0$. The Benamou-Brenier formula asserts that we have
\[
w_2(\rho_0,\rho_1)^2=\inf \left\{\int_0^1 \int_\Omega |v|^2\, \mu_t\,dx\,dt\, :\, \begin{array}{c}\partial_t \mu_t+\mathrm{div} (v_t\, \mu_t)=0 \\
\mu_0=\rho_0\ \mbox{ and }\ \mu_1=\rho_1
\end{array}\right\}.
\]
The latter consists in minimizing the integral of the total kinetic energy (the {\it action}), under a conservation of mass constraint. 
Thanks to the previous discussion, this dynamical problem can be transformed in a convex variational problem under linear constraint, once we introduce the variable
\[
\phi_t=v_t\, \mu_t.
\] 
For generalizations of this transport problem involving functions of the form $H(\phi)\, m^{-\beta}$ the reader can consult \cite{CCN} and \cite{DNS}.
\end{oss}

\begin{oss}[Sharpness of the condition on $\beta$]
The previous convexity property fails to be true in general for $\beta>p-1$. Let us fix $\phi_0\in\mathbb{R}^N\setminus\{0\}$ and $m_0>0$. We take $\phi_1=c\, \phi_0$ and $m_1=c\, m_0$ with $c>1$, then we consider the convex combination
\[
(m_t,\phi_t)=((1-t)\, m_0+t\, m_1, (1-t)\, \phi_0+t\, \phi_1).
\]
For every $p-1<\beta<p$ by strict concavity of the function $\tau\mapsto \tau^{p-\beta}$ we have
\[
\begin{split}
\frac{|\phi_t|^p}{m_t^\beta}=(1-t+t\,c)^{p-\beta}\, \frac{|\phi_0|^p}{m_0^\beta}&>(1-t)\,\frac{|\phi_0|^p}{m^\beta_0}+t\,c^{p-\beta}\, \frac{|\phi_0|^p}{m_0^\beta}\\
&=(1-t)\,\frac{|\phi_0|^p}{m_0^\beta}+\frac{|\phi_1|^p}{m_1^\beta}.
\end{split}
\] 
\end{oss}

\subsection{Hidden convexity}

The next convexity principle has been probably first identified by Benguria in his Ph.D. dissertation in the case $p=q=2$ and $H(z)=|z|^2$, see \cite{BBL}. See also \cite{BK} and \cite{TTU} for some generalizations in the case $q=p$. 
\begin{prop}[General hidden convexity]
\label{lm:hidden}
Let $1<p<\infty$ and $1<q\le p$. Let $H:\mathbb{R}^N\to[0+\infty)$ be a positively $p-$homogeneous convex function.  
For every pair of differentiable functions $u_0,u_1\ge 0$, we define
\[
\sigma_t(x)=\Big[(1-t)\, u_0(x)^q+t\, u_1(x)^q\Big]^\frac{1}{q}\qquad t\in[0,1],\,x\in\Omega.
\] 
Then there holds
\begin{equation}
\label{BK}
H(\nabla \sigma_t)\le (1-t)\,H(\nabla u_0)+t\, H(\nabla u_1),\quad t\in[0,1].
\end{equation}
\end{prop}
\begin{proof}
The proof for the case $p=q$ can be found for example in \cite{BF}. In order to consider the case $q<p$, we observe that by Lemma \ref{lm:convexity}, the function $F=H^{q/p}$ is a positively $q-$homogeneous convex function. Then the first part of the proof implies
\[
F(\nabla \sigma_t)\le (1-t)\, F(\nabla u_0)+t\, F(\nabla u_1),
\]
and by raising to the power $p/q$ and using the convexity of $\tau\mapsto \tau^{p/q}$, we end up with \eqref{BK}.
\end{proof}
\begin{oss}
We remark that neither strict convexity of $H$ nor strict positivity of the functions is needed in the previous result, unless one is interested in identification of equality cases in \eqref{BK}. Moreover, $H$ is only required to be only {\it positively} homogeneous, i.e. it is not necessarily even.
\end{oss}
\begin{oss}[Sharpness of the condition on $q$]
\label{oss:vital}
Again, the condition $q\le p$ is vital. Indeed, by taking a non-constant $u_0\not =0$ and $u_1=c\, u_0$ for $c>1$, then we have 
\[
\sigma_t(x)=((1-t)\, u_0^q+t\, u_1^q)^\frac{1}{q}=((1-t)+t\, c^q)^\frac{1}{q}\, u_0,
\]
and
\[
H(\nabla \sigma_t)=((1-t)+t\, c^q)^\frac{p}{q}\, H(\nabla u_0)>(1-t)\, H(\nabla u_0)+t\, H(\nabla u_1),
\]
by strict concavity of $\tau\mapsto \tau^{p/q}$.
\end{oss}

\subsection{Picone inequalities}
We now prove a general version of the so-called {\it Picone inequality}. The usual one, i.e.
\[
\left\langle|\nabla u|^{p-2}\, \nabla u, \nabla \left(\frac{v^p}{u^{p-1}}\right)\right\rangle\le |\nabla v|^p,\qquad v\ge 0,\, u>0,
\]
proved by Allegretto and Huang, see \cite[Theorem 1.1]{AH} corresponds to taking $p=q$ and $H(z)=|z|^p$ in \eqref{pitone} below.

\begin{prop}[General Picone inequality]
\label{lm:pitone}
Let $1<q\le p$ and let $H:\mathbb{R}^N\to[0,+\infty)$ be a $C^1$ positively $p-$homogeneous convex function. For every pair of positive differentiable functions $u,v$ with $u>0$, we have
\begin{equation}
\label{pitone}
\begin{split}
\frac{1}{p}\,\left\langle\nabla H(\nabla u), \nabla \left(\frac{v^q}{u^{q-1}}\right)\right\rangle\, &\le H(\nabla v)^\frac{q}{p}\, H(\nabla u)^\frac{p-q}{p}.
\end{split}
\end{equation}
\end{prop}
\begin{proof}
Let us start with the case $p=q$. We use the convexity inequality
\[
H(z)\ge H(w)+\langle \nabla H(w),z-w\rangle,
\]
with the choices
\[
z=\nabla v\qquad \mbox{ and }\qquad w=\nabla u\, \left(\frac{v}{u}\right).
\]
By using the $p-$homogeneity of $H$ we then get
\[
\begin{split}
H(\nabla v)&\ge \left(\frac{v}{u}\right)^p\, H(\nabla u)+\langle \nabla H(\nabla u),\nabla v\rangle\,\left(\frac{v}{u}\right)^{p-1}-\langle \nabla H(\nabla u),\nabla u\rangle\,\left(\frac{v}{u}\right)^p\\
&=\langle \nabla H(\nabla u),\nabla v\rangle\,\left(\frac{v}{u}\right)^{p-1}-\left(1-\frac{1}{p}\right)\, \langle \nabla H(\nabla u),\nabla u\rangle\,\left(\frac{v}{u}\right)^p\\
&=\frac{1}{p}\, \left\langle \nabla H(\nabla u),\nabla \left(\frac{v^p}{u^{p-1}}\right)\right\rangle, 
\end{split}
\]
which concludes the proof of \eqref{pitone} for $q=p$. 
\vskip.2cm\noindent
We now take $1<q<p$ and set
\[
F(z)=H(z)^\frac{q}{p},\qquad z\in\mathbb{R}^N,
\]
which is convex and positively $q-$homogeneous\footnote{Also observe that $F\in C^1(\mathbb{R}^N)$. It is sufficient to check that $F$ is differentiable at the origin and that its differential vanishes at $z=0$. Indeed, by using homogeneity we have
\[
F(h)-F(0)=H(h)^\frac{q}{p}=|h|^q\, H\left(\frac{h}{|h|}\right)^\frac{q}{p}=o(|h|),\qquad h\in\mathbb{R}^N\setminus\{0\} \mbox{ such that } |h|\ll 1.
\]}, then the first part of the proof implies
\[
\frac{1}{q}\,\left\langle \nabla F(\nabla u),\nabla\left(\frac{v^q}{u^{q-1}}\right)\right\rangle\le F(\nabla v).
\]
Observe that if $H(\nabla u)=0$, then we have $\nabla H(\nabla u)=0$ as well (see Lemma \ref{lm:analisi0}) and \eqref{pitone} holds true. Thus we can assume $H(\nabla u)\not =0$. The previous inequality is equivalent to
\[
\frac{1}{p}\, H(\nabla u)^\frac{q-p}{p}\, \left\langle \nabla H(\nabla u),\nabla\left(\frac{v^q}{u^{q-1}}\right)\right\rangle\le H(\nabla v)^\frac{q}{p}.
\]
If we multiply the previous by $H(\nabla)^{(p-p)/p}$ we eventually attains the conclusion.
\end{proof}
\begin{oss}[Equivalent form of the Picone inequality]
\label{oss:equivalent}
For $p\not =q$, as a plain consequence of Young inequality, \eqref{pitone} implies
\begin{equation}
\label{weak_pitone}
\frac{1}{p}\,\left\langle\nabla H(\nabla u), \nabla \left(\frac{v^q}{u^{q-1}}\right)\right\rangle\, \le \frac{q}{p}\,H(\nabla v)+\frac{p-q}{p}\, H(\nabla u).
\end{equation}
Observe that for $q=1$, the previous inequality reduces to
\[
H(\nabla u)+\langle \nabla H(\nabla u),\nabla v-\nabla u\rangle\le 
H(\nabla v),
\]
which just follows from the convexity of $z\mapsto H(z)$. On the other hand, by applying \eqref{weak_pitone} with the choices (here $\varepsilon>0$)
\[
U=\big(\varepsilon+H(\nabla u)\big)^{-\frac{1}{p}}\, u\qquad \mbox{ and }\qquad V=\big(\varepsilon+H(\nabla v)\big)^{-\frac{1}{p}}\, v,
\]
we get
\[
\frac{1}{p}\, \big(\varepsilon+H(\nabla u)\big)^{\frac{q-p}{p}}\, \big(\varepsilon+H(\nabla v)\big)^{-\frac{q}{p}}\, \left\langle\nabla H(\nabla u),\nabla\left(\frac{v^q}{u^{q-1}}\right) \right\rangle\le 1.
\]
By multiplying the previous by $(\varepsilon+H(\nabla u))^{(p-q)/p}\, (\varepsilon+H(\nabla v))^{\frac{q}{p}}$ and then letting $\varepsilon$ goes to $0$, we get \eqref{pitone}.
\end{oss}
\begin{oss}[Non-homogeneous functions]
All the convexity principles considered in this section have been proven under the assumption that $H$ is positively $p-$homogeneous. Nevertheless, the results are still true for some $H$ violating this condition. This is the case for example of the anisotopic function
\[
H(z)=\sum_{i=1}^N H_i(z_i),\qquad \mbox{ where }\quad H_i(t)=|t|^{p_i},
\]
and $1<p_1\le \dots\le p_N$. Namely, by applying \eqref{pitone} to each $H_i$ and then summing up, we get
\[
\sum_{i=1}^N |u_{x_i}|^{p_i-2}\, u_{x_i}\, \left(\frac{v^{q_i}}{u^{q_i-1}}\right)_{x_i}\le \sum_{i=1}^N |v_{x_i}|^\frac{q_i}{p_i}\, |u_{x_i}|^\frac{p_i-q_i}{p_i},
\]
for every $q_1,\dots,q_N$ such that $1<q_i\le p_i$. In the very same way from \eqref{BK} we get 
\[
\sum_{i=1}^N \big|(\sigma_t)_{x_i}\big|^{p_i}\le (1-t)\, \sum_{i=1}^N |u_{x_i}|^{p_i}+t\,\sum_{i=1}^N |v_{x_i}|^{p_i},\qquad t\in[0,1],
\]
where $\sigma_t=((1-t)\, u^q+t\, v^q)^{1/q}$ and $1<q\le p_1$. From Proposition \ref{lm:convessa} we can infer the convexity of the function
\[
(m,\phi)\mapsto \sum_{i=1}^N \frac{|\phi|^{p_i}}{m^{\beta_i}},\qquad (m,\phi)\in(0,+\infty)\times\mathbb{R}^N,
\]
for every $\beta_1,\dots,\beta_N$ such that $0<\beta_i\le p_i-1$.
\end{oss}

\section{Equivalences}
\label{sec:3}

In this section we will show that the three convexity principles proved in the previous section are indeed equivalent. In other words, they are just three different ways to look at the same principle. 

\subsection{Kinetic energies and Hidden convexity}

Let $u$ be a differentiable function on an open set $\Omega\subset\mathbb{R}^N$, which is everywhere positive. We first observe that if in the generalized kinetic energy of Proposition \ref{lm:convessa} we make the choice
\[
m=u\qquad \mbox{ and }\qquad \phi=\nabla u,
\]
then we obtain the functional
\[
u\mapsto \frac{H(\nabla u)}{u^\beta}.
\]
which is convex in the usual sense, provided that $0<\beta\le p-1$, thanks to Proposition \ref{lm:convessa}. This convexity is indeed equivalent to \eqref{BK}, as we now show.
\par
Indeed, let us pick two differentiable functions $u,v$ which are everywhere positive. We observe that by setting $U=u^q$, $V=v^q$ and $\gamma_t=(1-t)\, U+t\, V$, we get by Proposition \ref{lm:convessa}
\[
\frac{H(\nabla \gamma_t)}{\gamma_t^{\beta}}\le (1-t)\, \frac{H(\nabla U)}{U^\beta}+t\, \frac{H(\nabla V)}{V^\beta}.
\]
By using the homogeneity of $H$, the previous is equivalent to
\[
H\left(\nabla \gamma^\frac{p-\beta}{p}\right)\le (1-t)\, H\left(\nabla U^\frac{p-\beta}{p}\right)+t\, H\left(\nabla V^\frac{p-\beta}{p}\right).
\]
If we now choose $\beta$ in such a way that\footnote{Observe that such a choice is feasible, since 
\[
p\, \left(1-\frac{1}{q}\right)\le p-1 \quad \Leftrightarrow \quad \frac{p}{p-1}\le \frac{q}{q-1}\quad \Leftrightarrow \quad p\ge q.
\]}
\[
\frac{p-\beta}{p}=\frac{1}{q},\qquad \mbox{ i.e.}\quad \beta=p\,\left(1-\frac{1}{q}\right),
\]
we see that the previous inequality becomes
\[
H(\nabla \sigma_t)\le (1-t)\, H(u)+t\, H(v),
\]
where $\sigma_t=\gamma_t^{1/q}=((1-t)\, u^q+t\, v^q)^{1/q}$ as always.

\subsection{Hidden convexity and Picone}
We first show that
\[
(\mbox{Hidden convexity})\qquad \Longrightarrow\qquad (\mbox{Picone}).
\]
As before, given $u,v$ positive differentiable functions with $u>0$, we set
\begin{equation}
\label{curvetta}
\sigma_t(x)=\Big[(1-t)\, u(x)^q+t\, v(x)^q\Big]^\frac{1}{q}\qquad t\in[0,1],\,x\in\Omega.
\end{equation}
Then by using the convexity of $t\mapsto H(\nabla \sigma_t)$, we easily get
\[
\frac{H(\nabla\sigma_t)-H(\nabla u)}{t}\le H(\nabla v)-H(\nabla u).
\]
Observe that again by convexity, the incremental ratio on the left-hand side is monotone, then there exists the limit for $t$ monotonically converging to $0$, i.e. we obtain
\begin{equation}
\label{derivatina}
\left(\frac{d}{dt} H(\nabla \sigma_t)\right)_{|t=0}\le H(\nabla v)-H(\nabla u).
\end{equation}
The previous is indeed equivalent to \eqref{pitone}. To see this, it is sufficient to compute the derivative on the left-hand side. We have
\[
\nabla \sigma_t=\sigma_t^{1-q}\,\left[(1-t)\, \nabla u\, u^{q-1}+t\, \nabla v\, v^{q-1}\right]\qquad \mbox{ and }\qquad \frac{d}{dt} \sigma_t=\frac{1}{q}\,\sigma_{t}^{1-q}\, (v^q-u^q),
\]
so that we can compute
\[
\begin{split}
\frac{d}{dt} \nabla \sigma_t&=(1-q)\, \sigma_t^{-q}\, \left[(1-t)\, \nabla u\, u^{q-1}+t\, \nabla v\, v^{q-1}\right]\,\frac{d}{dt} \sigma_t\\
&+\sigma_t^{1-q}\, \left[\nabla v\, v^{q-1}-\nabla u\, u^{q-1}\right].
\end{split}
\]
Finally, we get
\[
\begin{split}
\left(\frac{d}{dt} \nabla\sigma_t\right)_{|t=0}&=-(q-1)\, \frac{\nabla u}{u}\,\left(\frac{d}{dt} \sigma_t\right)_{t=0}+\nabla v\, \left(\frac{v}{u}\right)^{q-1}-\nabla u\\
&=-\left(\frac{q-1}{q}\right)\, \nabla u\, \left(\left(\frac{v}{u}\right)^q-1\right)+\nabla v\, \left(\frac{v}{u}\right)^{q-1}-\nabla u.
\end{split}
\]
We can now compute the left-hand side of \eqref{derivatina} and obtain
\[
\begin{split}
\left(\frac{d}{dt} H(\nabla \sigma_t)\right)_{|t=0}&=\left\langle \nabla H(\nabla\sigma_t),\frac{d}{dt} \nabla\sigma_t\right\rangle_{|t=0}\\
&=\left\langle\nabla H(\nabla u), \nabla v\right\rangle\, \left(\frac{v}{u}\right)^{q-1}-\frac{p\,(q-1)}{q}\, H(\nabla u)\, \left(\frac{v}{u}\right)^q-\frac{p}{q}\, H(\nabla u).
\end{split}
\]
By inserting this in \eqref{derivatina}, observing that
\[
\frac{1}{q}\,\left\langle\nabla H(\nabla u), \nabla \left(\frac{v^q}{u^{q-1}}\right)\right\rangle=\left\langle \nabla H(\nabla u), \nabla v\right\rangle\, \left(\frac{v}{u}\right)^{q-1}-\frac{p\,(q-1)}{q}\, H(\nabla u)\, \left(\frac{v}{u}\right)^q,
\] 
and multiplying everything by $q/p$, we eventually get
\[
\frac{1}{p}\,\left\langle\nabla H(\nabla u), \nabla \left(\frac{v^q}{u^{q-1}}\right)\right\rangle\le \frac{q}{p}\,H(\nabla v)+\frac{p-q}{p}\, H(\nabla u).
\]
For $p=q$ this is exactly Picone inequality \eqref{pitone}, while for $q<p$ we just have to observe that by Remark \ref{oss:equivalent} the previous is equivalent to \eqref{pitone}.
\vskip.2cm\noindent
Let us now show that
\[
(\mbox{Picone})\qquad \Longrightarrow\qquad  (\mbox{Hidden convexity}).
\]
As we said, inequality \eqref{pitone} is actually equivalent to \eqref{derivatina},
for every $v,u$ and $\sigma_t$ curve of the form \eqref{curvetta} connecting them. We now fix $u,v$ and $\sigma_t$, then by \eqref{derivatina} we get
\[
H(\nabla v)-H(\nabla \sigma_t)\ge \frac{d}{ds} H(\nabla\widetilde\sigma_s)_{|s=0},
\]  
and
\[
H(\nabla u)-H(\nabla \sigma_t)\ge \frac{d}{ds} H(\nabla\widehat\sigma_s)_{|s=0},
\]
where $s\mapsto \widetilde \sigma_s$ and $s\mapsto \widehat \sigma_s$ are the curves of the form \eqref{curvetta} connecting $\sigma_t$ to $v$ and $\sigma_t$ to $u$ respectively. In other words, we have
\[
\widetilde \sigma_s=\Big[\left(1-t-s\,(1-t)\right)\, u^q+\left(t+s\,(1-t)\right)\, v^q\Big]^\frac{1}{q}=\sigma_{t+s\,(1-t)},\qquad s\in[0,1],
\]
and
\[
\widetilde \sigma_s=\Big[\left(1-t+s\,t\right)\, u^q+\left(t-s\,t \right)\, v^q\Big]^\frac{1}{q}=\sigma_{t-s\,t},\qquad s\in[0,1].
\]
Thus we get
\[
\begin{split}
H(\nabla v)-H(\nabla \sigma_t)\ge \frac{d}{ds} H(\nabla\widetilde\sigma_s)_{|s=0}&=\frac{d}{ds} H(\nabla \sigma_{t+s\,(1-t})_{|s=0}\\
&=\frac{d}{ds} H(\nabla \sigma_{s})_{|s=t}\, (1-t)
\end{split}
\]
and similarly
\[
H(\nabla u)-H(\nabla \sigma_t)\ge -\frac{d}{ds} H(\nabla \sigma_{s})_{|s=t}\, t.
\]
Keeping the two informations together, we finally get
\[
\frac{H(\nabla \sigma_t)-H(\nabla u)}{t}\le \frac{H(\nabla v)-H(\nabla\sigma_t)}{1-t},
\]
which is equivalent to $H(\nabla\sigma_t)\le (1-t)\, H(\nabla u)+t\, H(\nabla v)$.

\section{The discrete case}
\label{sec:4}

We now prove analogous results for functions which are not necessarily differentiable. Roughly speaking, we are going to replace derivatives by finite differences. For $1<p<\infty$ and $0<s<1$, the resulting convexity properties have applications to nonlocal integrals of the type
\[
\int_{\mathbb{R}^N} \int_{\mathbb{R}^N} \frac{|u(x)-u(y)|^p}{|x-y|^{N+s\,p}}\, dx\,dy,\qquad (\mbox{\it Gagliardo seminorm}),
\]
or
\[
\sup_{0<|h|} \int_{\mathbb{R}^N} \frac{|u(x+h)-u(x)|^p}{|h|^{s\,p}}\, dx,\qquad (\mbox{\it Nikolskii seminorm}),
\]
and more generally
\[
\int_{0}^{\infty} \left(\sup_{0<|h|\le t}\int_{\mathbb{R}^N} \frac{|u(x+h)-u(x)|^p}{t^{s\,p}}\, dx\right)^\frac{r}{p}\, \frac{dt}{t},\qquad p\le r<\infty,\quad (\mbox{\it Besov seminorm}).
\]
\begin{prop}[Discrete hidden convexity]
Let $1<p<\infty$ and $1<q\le p$. For every $u_0,u_1\ge 0$, we define
\[
\sigma_t(x)=\Big[(1-t)\, u_0(x)^q+t\, u_1(x)^q\Big]^\frac{1}{q}\qquad t\in[0,1],\,x\in\mathbb{R}^N.
\] 
Then we have
\begin{equation}
\label{discreta!}
|\sigma_t(x)-\sigma_t(y)|^p\le (1-t)\, |u_0(x)-u_0(y)|^p+t\, |u_1(x)-u_1(y)|^p,\quad t\in[0,1],\ x,y\in\mathbb{R}^N.
\end{equation}
\end{prop}
\begin{proof}
The proof is as in \cite{FP}, which deals with the case $p=q$. We observe that 
\[
\sigma_t=\left\|\left((1-t)^\frac{1}{q}\,u_0,t^\frac{1}{q}\,u_1\right)\right\|_{\ell^q},
\]
where we set $\|z\|_{\ell^q}=(|z_1|^q+|z_2|^q)^{1/q}$ for $z\in\mathbb{R}^2$. The triangular inequality implies that
\[
\big| \|z\|_{\ell^q}-\|w\|_{\ell^q}\big|^q\le \|z-w\|_{\ell^q}^q,\qquad z,w\in\mathbb{R}^2, 
\] 
and by using this with the choices
\[
z=\left((1-t)^\frac{1}{q}\,u_0(x),t^\frac{1}{q}\,u_1(x)\right)\qquad \mbox{ and }\qquad w=\left((1-t)^\frac{1}{q}\,u_0(y),t^\frac{1}{q}\,u_1(y)\right),
\]
we get
\[
|\sigma_t(x)-\sigma_t(y)|^q\le (1-t)\, |u_0(x)-u_0(y)|^q+t\, |u_1(x)-u_1(y)|^q.
\]
By raising both sides to the power $p/q$ and using the convexity of $\tau\mapsto \tau^{p/q}$, we get \eqref{discreta!}.
\end{proof}

\begin{prop}[Discrete Picone inequality]
\label{prop:disc_pic}
Let $1<p<\infty$ and $1<q\le p$.
Let $u,v$ be two measurable functions with $v\ge 0$ and $u>0$, then
\begin{equation}
\label{pitone_frazione}
\begin{split}
|u(x)-u(y)|^{p-2}\, (u(x)-u(y))\,& \left[\frac{v(x)^q}{u(x)^{q-1}}-\frac{v(y)^q}{u(y)^{q-1}}\right]\\
&\le|v(x)-v(y)|^q\,|u(x)-u(y)|^{p-q}.
\end{split}
\end{equation}
\end{prop}
\begin{proof}
We notice at first that is sufficient to prove
\begin{equation}
\label{picone_frazione}
\begin{split}
|u(x)-u(y)|^{q-2}\, (u(x)-u(y))\,& \left[\frac{v(x)^q}{u(x)^{q-1}}-\frac{v(y)^q}{u(y)^{q-1}}\right]\le|v(x)-v(y)|^q,
\end{split}
\end{equation}
since \eqref{pitone_frazione} then follows by multiplying the previous inequality by $|u(x)-u(y)|^{p-q}$.
\par
At this aim, let us start by observing that if $u(x)=u(y)$, inequality \eqref{picone_frazione} is trivially satisfied. We take then $u(x)\not=u(y)$ and we can always suppose that $u(x)<u(y)$, up to exchanging the role of $x$ and $y$. We further observe that if $v(y)=0$, inequality \eqref{picone_frazione} is again trivially satisfied, since 
\[
|u(x)-u(y)|^{q-2}\, (u(x)-u(y))\left[\frac{v(x)^q}{u(x)^{q-1}}-\frac{v(y)^q}{u(y)^{q-1}}\right]\le 0.
\]
We can thus suppose that $v(y)\not =0$, then we rewrite the left-hand side of \eqref{picone_frazione} as
\[
\begin{split}
|u(x)-u(y)|^{q-2}\,& (u(x)-u(y)) \left[\frac{v(x)^q}{u(x)^{q-1}}-\frac{v(y)^q}{u(y)^{q-1}}\right]\\
&=u(x)^q\,\left(\frac{v(y)}{u(y)}\right)^q\left[ \left(1-\frac{u(y)}{u(x)}\right)^{q-1}\, \left(\left(\frac{v(x)\,u(y)}{v(y)\,u(x)}\right)^q-\frac{u(y)}{u(x)}\right)\right]
\end{split}
\] 
while the right-hand side of \eqref{picone_frazione} rewrites as
\[
\begin{split}
|v(x)-v(y)|^q=u(x)^q\,\left(\frac{v(y)}{u(y)}\right)^q\,\left|\left(\frac{v(x)\,u(y)}{v(y)\,u(x)}\right)-\frac{u(y)}{u(x)}\right|^q.
\end{split}
\]
Then if we set
\[
A=\frac{v(x)\,u(y)}{v(y)\,u(x)}\qquad \mbox{ and }\qquad t=\frac{u(y)}{u(x)},
\]
the previous manipulations show that \eqref{pitone_frazione} is equivalent to the following
\[
(1-t)^{q-1}\, (A^q-t)\le |A-t|^{q},\qquad \mbox{ for } 0\le t\le 1,
\]
The previous elementary inequality is true (see \cite[Lemma 2.6]{FS}), thus we get the desired conclusion.
\end{proof}

\section{Some applications: local integrals}\label{sec:local}
\label{sec:5}

\subsection{Positive eigenfunctions}

In what follows, we denote by $\Omega\subset\mathbb{R}^N$ an open connected set such that $|\Omega|<\infty$. For $1<p<\infty$, as it is customary we denote by $W^{1,p}_0(\Omega)$ the closure of $C^\infty_0(\Omega)$ with respect to the $L^p$ norm of the gradient. We also take $H:\mathbb{R}^N\to [0,\infty)$ to be a $C^1$ convex $p-$homogeneous function such that 
\[
\frac{1}{C}\, |z|^p\le H(z)\le C\, |z|^p,\qquad z\in\mathbb{R}^N,
\]
for some $C\ge 1$.
Then for $1<q\le p$, we set
\begin{equation}
\label{autolavoro}
\lambda_{p,q}(\Omega)=\min_{u\in W^{1,p}_0(\Omega)}\left\{\int_\Omega H(\nabla u)\, dx\, :\, \|u\|_{L^q(\Omega)}=1\right\}.
\end{equation}
\begin{thm}[Uniqueness of positive eigenfunctions]
\label{teo:eigen}
Let $1<p<\infty$ and $1<q\le p$. Let $\lambda>0$ be such that there exists a non trivial function $u\in W^{1,p}_0(\Omega)$ verifying
\[
-\frac{1}{p}\,\mathrm{div\,} \nabla H(\nabla u)=\lambda\,  u^{q-1},\qquad u\ge 0,\quad \mbox{ in }\Omega.
\]
Then we have
\begin{equation}
\label{dopolavoro_ferroviario}
\lambda\, \left(\int_\Omega |u|^q\, dx\right)^\frac{q-p}{q}=\lambda_{p,q}(\Omega),
\end{equation}
and $v=u\,\|u\|_{L^q(\Omega)}^{-1}$ is a minimizer of the variational problem in \eqref{autolavoro}.
\end{thm}
\begin{proof}
We first observe that $u>0$ almost everywhere\footnote{Actually, the much stronger result
\[
\inf_K u\ge \frac{1}{C_K},\qquad \mbox{ for every compact }K\Subset\Omega,
\]
holds. Here we want to point out that the weaker information ``{\it $u>0$ almost everywhere}'' suffices for this argument to work.} in $\Omega$, by the strong minimum principle (see for example \cite[Theorem 1.2]{Tr}). We also notice that the case $p=q$ is now well-established (see for example \cite{AH,BF,Ja,TTU}), we limit ourselves to consider the case $q<p$.
\vskip.2cm\noindent
The proof is just based on an application of Proposition \ref{lm:pitone}. We observe that $v$ is a solution of 
\[
-\frac{1}{p}\,\mathrm{div\,}\nabla H(\nabla v)=\lambda\,\|u\|_{L^q(\Omega)}^{q-p}\, v^{q-1},\qquad v> 0,\quad \mbox{ in }\Omega.
\]
Moreover, $v$ is admissible for the variational problems defining $\lambda_{p,q}(\Omega)$, thus by testing the previous equation with $v$ itself and using the homogeneity of $H$, we get
\[
\lambda\, \|u\|_{L^q(\Omega)}^{q-p}\ge \lambda_{p,q}(\Omega).
\]
Let $u_1\in W^{1,p}_0(\Omega)$ be a function achieving the minimum in the right-hand side of \eqref{autolavoro}. Then we have
\[
\begin{split}
\lambda\,\|u\|_{L^p(\Omega)}^{q-p}=\lambda\,\|u\|_{L^p(\Omega)}^{q-p}\,\int_\Omega u_1^{q}\, dx&=\lambda\,\|u\|_{L^p(\Omega)}^{q-p}\, \int_\Omega v^{q-1}\, \frac{u^q_1}{v^{q-1}}\, dx\\
&=\frac{1}{p}\,\int_\Omega \left\langle \nabla H(\nabla v),\nabla\left(\frac{u^q_1}{v^{q-1}}\right)\right\rangle\, dx\\
&\le \int_\Omega H(\nabla u_1)^\frac{q}{p}\, H(\nabla v)^\frac{p-q}{p}\, dx.
\end{split}
\]
If we now apply H\"older's and Young's inequality, the previous gives the desired result.
\end{proof}
\begin{oss}
Of course, a completely equivalent proof of the previous result could use Proposition \ref{lm:hidden}, as in \cite{BF}. As remarked in the Introduction, we believe that a direct application of hidden convexity provides a cleaner justification of the result, while on the other hand Picone inequality offers a quicker proof. An alternative proof can be found in \cite{OT}, later refined by Kawohl and Lindqvist in \cite{KL}.  
\end{oss}
\begin{oss}[Sharpness of the condition $q\le p$]
For a general open set $\Omega$ with finite measure, the previous result {\it can not hold true for $q>p$}. Indeed, let us consider an annular domain $T=\{x\in\mathbb{R}^N\, :\, 1<|x|<r\}$ and take $H(z)=|z|^p$, then the problem
\[
\lambda^{rad}_{p,q}(T)=\min_{W^{1,p}_0(T)}\left\{\int_T |\nabla u|^p\, dx\, :\, u \mbox{ radial function},\|u\|_{L^q(T)}=1\right\},
\]
admits a minimizer $u_0\in W^{1,p}_0(T)$, which is a {\it positive} solution of
\[
-\Delta_p u=\lambda^{rad}_{p,q}(T)\, u^{q-1},\qquad \mbox{ in }T,\qquad \ \mbox{ with }\ \int_T |u|^q\, dx=1.
\] 
On the other hand, Nazarov in \cite[Proposition 1.2]{Na} has proved that if $q>p$ one can always take $r$ sufficiently close to $1$ such that minimizers of \eqref{autolavoro} are not radial. This clearly means that
\[
\lambda^{rad}_{p,q}(T)>\lambda_{p,q}(T).
\]
\end{oss}
\subsection{Hardy-type inequalities}\label{ssec:hardyloc}
As another application of the general Picone inequality \eqref{pitone}, we have the following family of sharp inequalities.
\begin{thm}[Weighted Hardy inequalities with general norms]
\label{thm:hardygen}
Let $F:\mathbb{R}^N\to[0,+\infty)$ be a $C^1$ strictly convex norm. Let $1<p<N$, for every $\gamma> p-N$ we have
\begin{equation}
\label{hardygen}
\left(\frac{N+\gamma-p}{p}\right)^p\,\int_{\mathbb{R}^N} |v|^p\,F_*(x)^{\gamma-p}\, dx\le \int_{\mathbb{R}^{N}} F(\nabla v)^p\,F_*(x)^{\gamma}\, dx, \ \ v\in C^1_0(\mathbb{R}^N\setminus\{0\}),
\end{equation}
where $F_*$ is the dual norm defined by
\[
F_*(z)=\sup_{x\not=0} \left\langle\frac{x}{F(x)},z\right\rangle,\qquad z\in\mathbb{R}^N.
\]
\end{thm}
\begin{proof}
We start proving \eqref{hardygen} for positive $C^1_0(\mathbb{R}^N\setminus\{0\})$ functions. We take $\beta>0$ and set 
\[
u(x)=F_*(x)^{-\beta},\qquad x\in\mathbb{R}^N\setminus\{0\}.
\]
Observe that $u$ is a $C^1$ function in $\mathbb{R}^N\setminus\{0\}$ such that
\begin{equation}
\label{merdatotale}
\begin{split}
-\mathrm{div}(F_*(x)^\gamma&\,F(\nabla u)^{p-1}\,\nabla F(\nabla u))\\
&=\beta^{p-1}\,\Big[N-p-\beta\, (p-1)+\gamma\Big]\, u^{p-1}\, F_*(x)^{\gamma-p},\quad \mbox{ in }\mathbb{R}^{N}\setminus\{0\}.
\end{split}
\end{equation}
Indeed, we recall the following relations between $F$ and $F_*$ (see \cite{Sc} for example)
\begin{equation}
\label{magic}
F(\nabla F_*(x))=1\qquad \mbox{ and }\qquad \nabla F(\nabla F_*(x))=\frac{x}{F_*(x)},\qquad x\not=0.
\end{equation}
Of course
\[
\nabla u=-\beta\, F_*(x)^{-\beta-1}\,\nabla F_*(x),\qquad x\not=0,
\]
by using \eqref{magic} and the homogeneity of $F$ we get
\[
F(\nabla u)^{p-1}=\beta^{p-1}\, F_*(x)^{-(\beta+1)\,(p-1)}=\beta^{p-1}\, F_*(x)^{-\beta\,p+\beta-p},
\]
and still by \eqref{magic} and the fact that $\nabla F$ is $0-$homogeneous, we also have
\[
\nabla F(\nabla u)=-\nabla F(\nabla F_*(x))=-\frac{x}{F_*(x)},\qquad x\not=0.
\]
Thus we get
\[
\begin{split}
-\mathrm{div}\left(F_*(x)^\gamma\,F(\nabla u)^{p-1}\, \nabla F(\nabla u)\right)&=\beta^{p-1}\, \mathrm{div}\left(F_*(x)^{-\beta\,p+\beta-p+\gamma}\, x\right)\\
&=\beta^{p-1}\, [N-\beta\,p+\beta-p+\gamma]\, F_*(x)^{-\beta\,(p-1)+\gamma-p},
\end{split}
\]
as desired, where we used that
\[
\langle \nabla F_*(x),x\rangle=F_*(x),\qquad x\not=0,
\]
again by homogeneity. Finally, by using the definition of $u$, we get 
\[
F_*(x)^{-\beta\,(p-1)+\gamma-p}=u^{p-1}\, F_*(x)^{\gamma-p}.
\]
Then $u$ verifies
\[
\begin{split}
C_{N,p}\, \int_{\mathbb{R}^N} u^{p-1}\, &F_*(x)^{\gamma-p}\,\varphi\, dx\\
&=\int_{\mathbb{R}^N} F_*(x)^\gamma\,F(\nabla u)^{p-1}\,\langle \nabla F(\nabla u),\nabla\varphi\rangle\, dx,\qquad \varphi\in C^1_0(\mathbb{R}^N\setminus\{0\}).
\end{split}
\]
We test the previous equation with $\varphi=v^{p}\, u^{1-p}$, where $v\in C^1_0(\mathbb{R}^N\setminus\{0\})$ is positive. We get
\[
\begin{split}
\beta^{p-1}\,\Big[N-\beta\,p+\beta-p+\gamma\Big]&\int_{\mathbb{R}^N} v^p\,F_*(x)^{\gamma-p}\, dx\\
&=\int_{\mathbb{R}^N} F_*(x)^\gamma\,F(\nabla u)^{p-1}\left\langle \nabla F(\nabla u),\nabla\left(\frac{v^p}{u^{p-1}}\right)\right\rangle\, dx.
\end{split}
\]
We can now use Proposition \ref{lm:pitone} with the choice $H(z)=F(z)^p$,
so to obtain
\[
\beta^{p-1}\,\Big[N-\beta\,p+\beta-p+\gamma\Big]\,\int_{\mathbb{R}^N} v^p\,F_*(x)^{\gamma-p}\, dx\le  \int_{\mathbb{R}^N} F(\nabla v)^p\, F_*(x)^{\gamma} \, dx
\]
In order to conclude, it is now sufficient to observe that the function
\[
\beta\mapsto \beta^{p-1}\,\Big[N-\beta\, p+\beta-p+\gamma\Big],
\]
is maximal for $\beta=(N+\gamma-p)/p$. The previous argument gives \eqref{hardygen} for a positive $v\in C^1_0(\mathbb{R}^N\setminus\{0\})$. Of course, the previous proof is still valid for positive Lipschitz functions supported in $\mathbb{R}^N\setminus\{0\}$.
The result for a general $v\in C^1_0(\mathbb{R}^N\setminus\{0\})$ then follows by writing $v=v_+-v_-$ and 
observing that $v_+,v_-$ are positive Lipschitz functions with support in $\mathbb{R}^N\setminus\{0\}$. 
\end{proof}
\begin{oss}
By taking $\gamma=0$ in \eqref{hardygen}, we have the usual Hardy inequality on the whole space with respect to a general norm, i.e.
\begin{equation}
\label{hardyvanbasten}
\left(\frac{N-p}{p}\right)^p\, \int_{\mathbb{R}^N} \left(\frac{|v|}{F_*(x)}\right)^p\, dx\le \int_{\mathbb{R}^N} F(\nabla v)^p\, dx
\end{equation}
 A different proof of \eqref{hardyvanbasten} (based on symmetrization techniques) can be found in \cite[Proposition 7.5]{Van}. For $\gamma\not =0$ and $F$ being the Euclidean norm, a related inequality can be found in \cite{ACP}.
\end{oss}

\section{Some applications: nonlocal integrals}
\label{sec:6}

\subsection{Positive eigenfunctions}
We denote by $\Omega\subset\mathbb{R}^N$ an open connected set, which is now supposed to be bounded. 
Let $1<p<\infty$ and $0<s<1$, in what follows we denote by $W^{s,p}_0(\Omega)$ the completion of $C^\infty_0(\Omega)$ with respect to the norm
\[
\|u\|_{W^{s,p}_0(\Omega)}=\left(\int_{\mathbb{R}^N}\int_{\mathbb{R}^N} \frac{|u(x)-u(y)|^{p}}{|x-y|^{N+s\,p}}\, dx\, dy\right)^\frac{1}{p}.
\]
As before, for $1<q\le p$ we introduce the first eigenvalue
\begin{equation}
\label{autolavoro_s}
\lambda^s_{p,q}(\Omega)=\min_{u\in W^{s,p}_0(\Omega)}\left\{\|u\|^p_{W^{s,p}_0(\Omega)}\, :\, \|u\|_{L^q(\Omega)}=1\right\},
\end{equation}
the reader is referred to \cite{BLP,FP,LL} for a more detailed account about the case $q=p$.
We notice that a minimizer $u$ of the previous problem is a weak solution of
\[
(-\Delta_p)^s u=\lambda^s_{p,q}(\Omega)\,  |u|^{q-2}\, u,\qquad \mbox{ in }\Omega,
\]
which means that 
\[
\int_{\mathbb{R}^N}\int_{\mathbb{R}^N} \frac{|u(x)-u(y)|^{p-2}\, \big(u(x)-u(y)\big)}{|x-y|^{N+s\,p}}\, \big(\varphi(x)-\varphi(y)\big)\, dx\, dy=\lambda^s_{p,q}(\Omega)\, \int_\Omega |u|^{q-2}\, u\, \varphi\, dx,
\]
for every $\varphi\in W^{s,p}_0(\Omega)$.
\begin{thm}[Uniqueness of positive eigenfunctions]
Let $1<p<\infty$, $0<s<1$ and $1<q\le p$. Let $\lambda>0$ be such that there exists a non trivial function $u\in W^{s,p}_0(\Omega)$ verifying
\[
(-\Delta_p)^s u=\lambda\,  |u|^{q-2}\, u,\qquad u\ge 0,\quad \mbox{ in }\Omega.
\]
Then we have
\begin{equation}
\label{sdopolavoro_ferroviario}
\lambda\, \left(\int_\Omega |u|^q\, dx\right)^\frac{q-p}{q}=\lambda_{p,q}(\Omega),
\end{equation}
and $v=u\,\|u\|_{L^q(\Omega)}^{-1}$ is a minimizer of the problem in \eqref{autolavoro_s}.
\end{thm}
\begin{proof}
At first, it is again crucial to observe that $u>0$ almost everywhere in $\Omega$, thanks to the minimum principle of Theorem \ref{teo:minprin}. 
Then the result for the case $p=q$ follows by using \cite[Theorem 4.1]{FP}. We now consider the case $q<p$.
\vskip.2cm\noindent
Again, we observe that $v$ solves
\[
(-\Delta_p)^s v=\lambda\,\|u\|_{L^q(\Omega)}^{q-p}\, v^{q-1},\qquad u> 0,\quad \mbox{ in }\Omega.
\]
and since $v$ is admissible for the variational problems defining $\lambda^s_{p,q}(\Omega)$ we get
\[
\lambda\, \|u\|_{L^q(\Omega)}^{q-p}\ge \lambda^s_{p,q}(\Omega).
\]
Let $u_1\in W^{s,p}_0(\Omega)$ be a function achieving the minimum in the right-hand side of \eqref{autolavoro_s}. Then again we have
\[
\begin{split}
\lambda\,\|u\|_{L^q(\Omega)}^{q-p}&=\lambda\,\|u\|_{L^q(\Omega)}^{q-p}\, \int_\Omega v^{q-1}\, \frac{u^q_1}{v^{q-1}}\, dx\\
&=\int_{\mathbb{R}^N} \int_{\mathbb{R}^N} \frac{|v(x)-v(y)|^{p-2}\, (v(x)-v(y))}{|x-y|^{N+s\,p}}\, \left[\frac{u_1(x)^q}{v(x)^{q-1}}-\frac{u_1(y)^q}{v(y)^{q-1}}\right]\, dx\,dy\\
&\le \int_{\mathbb{R}^N} \int_{\mathbb{R}^N} \frac{|u_1(x)-u_1(y)|^q}{|x-y|^{N\, \frac{q}{p}+s\,q}}\,\frac{|v(x)-v(y)|^{p-q}}{|x-y|^{N\,\frac{p-q}{p}+s\,(p-q)}}\, dx\, dy \\
\end{split}
\]
where we used Proposition \ref{prop:disc_pic}.
If we now apply H\"older's and Young's inequalities with exponents $p/q$ and $p/(p-q)$, the previous gives the desired result.
\end{proof}

\subsection{Hardy-type inequalities}

As in the local case, by means of the discrete Picone inequality \eqref{pitone_frazione} we can prove a nonlocal Hardy inequality, like in \cite{FS}. The idea is still to look at power-type positive solutions of
\[
(-\Delta_p)^s u=\lambda\, u^{p-1},\qquad \mbox{ in }\mathbb{R}^N\setminus\{0\}.
\]
The latter is the content of the next technical result.
\begin{lm}
\label{lm:coddio}
Let $1<p<\infty$ and $\beta,r>0$. Then the function
\[
f(x)=\int_{\mathbb{R}^N} \frac{\left||x|^{-\beta}-|y|^{-\beta}\right|^{p-2}\,\left(|x|^{-\beta}-|y|^{-\beta}\right)}{|x-y|^{N+r}}\, dy,\qquad x\in\mathbb{R}^N,
\]
is radial and $\beta\,(1-p)-r$ homogeneous.
\end{lm}
\begin{proof}
Let us pick $x_1\not=x_2$ such that $|x_1|=|x_2|$. Let us set call $R:\mathbb{R}^N\to\mathbb{R}^N$ the linear isometry defined by the reflection in the hyperplan
\[
\pi=\left\{x\in\mathbb{R}^N\, :\, \langle x,x_1-x_2\rangle=0\right\}.
\]
Then by changing variables we have
\[
\begin{split}
f(x_1)&=\int_{\mathbb{R}^N} \frac{\left||x_1|^{-\beta}-|R\,z|^{-\beta}\right|^{p-2}\, \left(|x_1|^{-\beta}-|R\,z|^{-\beta}\right)}{|x_1-R\, z|^{N+r}}\, dz\\
&=\int_{\mathbb{R}^N} \frac{\left||x_2|^{-\beta}-|R\,z|^{-\beta}\right|^{p-2}\, \left(|x_2|^{-\beta}-|R\,z|^{-\beta}\right)}{|R\, x_2-R\, z|^{N+r}}\, dz\\
&=\int_{\mathbb{R}^N} \frac{\left||x_2|^{-\beta}-|z|^{-\beta}\right|^{p-2}\,\left(|x_2|^{-\beta}-|z|^{-\beta}\right)}{|x_2-z|^{N+r}}\, dz=f(x_2),
\end{split}
\]
where we used that $R\, x_2=x_1$, that $|R\, z|=|z|$ and the linearity of $R$. This shows that $f$ is radial. 
\vskip.2cm\noindent
For the second part, it is sufficient to observe that
\[
\begin{split}
f(t\, x)&=\int_{\mathbb{R}^N} \frac{\left|t^{-\beta}\,|x|^{-\beta}-|y|^{-\beta}\right|^{p-2}\, \left(t^{-\beta}\, |x|^{-\beta}-|y|^{-\beta}\right)}{|t\,x-y|^{N+r}}\, dy\\
&=t^{-\beta\,(p-1)-N-r}\,\int_{\mathbb{R}^N} \frac{\left||x|^{-\beta}-|z|^{-\beta}\right|^{p-2}\,(|x|^{-\beta}-|z|^{-\beta})}{|x-z|^{N+r}}\, t^N\, dz\\
&=t^{-\beta\,(p-1)-r} f(x),\qquad x\in\mathbb{R}^N,
\end{split}
\]
for all $t>0$, which gives the desired conclusion.
\end{proof}
We then have the following sharp Hardy inequality for the fractional Sobolev space $W^{s,p}$, first proved in \cite{FS}.
\begin{thm}[Fractional Hardy inequality]\label{thm:hardyfrc}
Let $s\in(0,1)$ and $1<p<\infty$ such that $s\,p<N$. Then there exists a constant $C=C(N,s,p)>0$ (see equation \eqref{costante} below) such that
\begin{equation}
\label{hardy_frazia}
C\,\int_{\mathbb{R}^N} \frac{|v|^p}{|x|^{s\,p}}\,dx \le \int_{\mathbb{R}^N}\int_{\mathbb{R}^N}  \frac{\big| v(x)-v(y)\big|^p}{|x-y|^{N+s\,p}}\,dx\,dy,
\end{equation}
for all $v\in C^\infty_0(\mathbb{R}^N\setminus\{0\})$.
\end{thm}
\begin{proof}
We recall that a positively $\alpha-$homogeneous function $u$ which is radially simmetric, that is
\[
u(x) = \varphi(|x|)\, \qquad x\in\mathbb{R}^N\,,
\]
is uniquely determined modulo a multiplicative constant, namely $u(x)=\varphi(1)\, |x|^\alpha$. By this elementary
observation, we can deduce from Lemma \ref{lm:coddio} that the function $u(x)=|x|^{-\beta}$ is a solution of
\[
\begin{split}
\int_{\mathbb{R}^N}\int_{\mathbb{R}^N} &\frac{\left|u(x)-u(y)\right|^{p-2}\,\left(u(x)-u(y)\right)}{|x-y|^{N+s\,p}}\, (\varphi(x)-\varphi(y))\, dx\,dy\\
&= C(\beta)\, \int_{\mathbb{R}^N} u^{p-1}\, |x|^{-s\,p}\, \varphi(x)\, dx, \qquad \mbox{ for all } \varphi\in C^\infty_0(\mathbb{R}^N\setminus\{0\}),
\end{split}
\]
where
\[
C(\beta)=2\,\int_{\mathbb{R}^N} \frac{\left||x|^{-\beta}-|y|^{-\beta}\right|^{p-2}\,\left(|x|^{-\beta}-|y|^{-\beta}\right)}{|x-y|^{N+s\,p}}\, dy,\qquad x\in\mathbb{S}^{N-1},
\]
and the previous integral is constant for $x\in\mathbb{S}^{N-1}$, thanks to Lemma \ref{lm:coddio}.
Then, if one picks $v\in C^\infty_0(\mathbb{R}^N\setminus\{0\})$ positive and plugs in $\varphi = v^p\,u^{1-p}$ as a test function in the previous equation, it turns out that
\begin{equation}
\label{belliffima}
C(\beta)\, \int_{\mathbb{R}^N} |v|^p\, |x|^{-s\,p}\,dx\le\int_{\mathbb{R}^N} \int_{\mathbb{R}^N}\frac{|v(x)-v(y)|^p}{|x-y|^{N+sp}}\,dx\,dy.
\end{equation}
An optimization over $\beta$ leads to the desired result, see Appendix B for more details.
\end{proof}

\appendix

\section{A minimum principle for positive nonlocal eigenfunctions}
\label{sec:A}

In the following, we provide a proof of a minimum principle for positive weak supersolutions
to equation
\[
(-\Delta_p)^s u=0,\quad \mbox{ in }\Omega,\qquad\qquad u\equiv 0\quad \mbox{ in }\mathbb{R}^N\setminus\Omega.
\]
i.e. for functions $u\in W^{s,p}_0(\Omega)$ such that
\[
\int_{\mathbb{R}^N}\int_{\mathbb{R}^N} \frac{|u(x)-u(y)|^{p-2}(u(x)-u(y))}{|x-y|^{N+sp}}\,(\varphi(x)-\varphi(y))\, dx\,dy \ge 0, \quad \mbox{ for all} \ \varphi \in C_0^\infty(\Omega), \ \varphi\ge0.
\]
Let $x_0$ be any fixed point in $\Omega$, and for every $r>0$ let $B_r(x_0)$ denote the ball of radius $r$ centered at $x_0$. The main ingredient of our result is the following lemma, 
which is a consequence of a more general logarithmic estimate recently established by Di Castro, Kuusi and Palatucci in \cite{DKP}.
\begin{dkp}
Let $1<p<\infty$, $s\in(0,1)$ and $u\in W^{s,p}_0(\Omega)$ be a supersolution such that $u\ge 0$ in $B_{2\,r}(x_0)\Subset\Omega$. Then for every $0<\delta<1$ there holds
\begin{equation}
\begin{split}
\label{Cacciotail}
\int_{B_r} \int_{B_r} \left|\log\left(\frac{\delta+u(x)}{\delta+u(y)}\right)\right|^p&\frac{1}{|x-y|^{N+s\,p}} \, dx\,dy\\
& \le C\, r^{N-s\,p} \left\{\delta^{1-p}\, r^{s\,p}\,\int_{\mathbb{R}^N\setminus B_{2\,r} } \frac{u_-(y)^{p-1}}{|y-x_0|^{N+sp}}\,dy+ 1\right\},
\end{split}
\end{equation}
where $u_-=\max\{-u,0\}$ and $C=C(N,p,s)>0$ is a constant.
\end{dkp}
We then have the following minimum principle. 
\begin{thm}
\label{teo:minprin}
Let $\Omega\subset\mathbb{R}^N$ be an open bounded set, which is connected. Let $s\in(0,1)$, $1<p<\infty$ and $u\in W^{s,p}_0(\Omega)$ be a weak supersolution such that $u\ge 0$ in $\Omega$. 
Let us suppose that
\begin{equation}
\label{ipotesi}
u\not\equiv 0\qquad \mbox{ in } \Omega.
\end{equation}
Then $u>0$ almost everywhere in $\Omega$.
\end{thm}
\begin{proof}
We first prove that for every $K\Subset\Omega$ compact connected, if 
\begin{equation}
\label{ipotesi2}
u\not\equiv 0\qquad \mbox{ in } K,
\end{equation}
then $u>0$ almost everywhere in $K$.
\par
Let $K\Subset\Omega$ be a connected compact set, then $K\subset\{x\in\Omega\, :\, \mathrm{dist}(x,\partial\Omega)>2\,r\}$ for some $r>0$. We then observe that $K$ can be covered by a finite number of balls $B_{r/2}(x_1),\dots B_{r/2}(x_k)$ such that $x_i\in K$ and
\begin{equation}
\label{incatenate}
|B_{r/2}(x_i)\cap B_{r/2}(x_{i+1})|>0,\quad i=1,\dots, k-1.
\end{equation}
Let us now suppose that $u=0$ on a subset of $K$ with positive measure. Then for some $i\in\{1,\dots,k-1\}$, we have that the set 
\[
Z:=\{x\in B_{r/2}(x_i)\,: \, u(x)=0\},
\]
has positive measure. We define
\[
F_\delta(x)=\log\left(1+\frac{u(x)}{\delta}\right),\qquad x\in B_{r/2}(x_i),
\]
for $\delta>0$ and claim that the following Poincar\'e inequality holds true
\begin{equation}
\label{sobs}
\int_{B_{r/2}(x_i)} |F_\delta|^p\, dx\le  \frac{r^{N+s\,p}}{|Z|}\,
\int_{B_{r/2}(x_i)}\int_{B_{r/2}(x_i)}\frac{|F_\delta(x)-F_\delta(y)|^p}{|x-y|^{N+s\,p}}\, dx\,dy.
\end{equation}
Indeed, observe that 
\[
F_\delta(x)=0\qquad \mbox{ for every } x\in Z\,,
\]
hence for every $x\in B_{r/2}(x_i)$ and $y\not=x$ with $y\in Z$, we get
\[
|F_\delta(x)|^p=\frac{|F_\delta(x)-F_\delta(y)|^p}{|x-y|^{N+s\,p}}\,|x-y|^{N+s\,p}\,.
\]
Now integrating with respect to $y\in Z$ gives
\[
|Z|\, |F_\delta(x)|^p\le\left(\max_{x,y\in B_{r/2}(x_i)} |x-y|^{N+s\,p}\right)\,\int_{B_{r/2}(x_i)}\frac{|F_\delta(x)-F_\delta(y)|^p}{|x-y|^{N+s\,p}}\, dy\,,
\]
which proves \eqref{sobs} up to an integration with respect to $x\in B_{r/2}(x_i)$.
\vskip.2cm\noindent
We now observe that 
\[
\left|\log\left(\frac{\delta+u(x)}{\delta+u(y)}\right)\right|^p=
\left| F_\delta(x) - F_\delta(y)\right|^p,
\]
thus if we combine \eqref{sobs}, \eqref{Cacciotail} and observe that $u_-\equiv 0$, we get
\begin{equation}
\label{BMO}
\begin{split}
\int_{B_{r/2}(x_i)} \left|\log\left(1+\frac{u}{\delta}\right)\right|^p&\, dx\le  C\, \frac{r^{2\,N}}{|Z|},
\end{split}
\end{equation}
with $C$ independent of $\delta$.
By letting $\delta$ go to $0$ in \eqref{BMO}, we can then infer 
\[
u=0\qquad \mbox{ almost everywhere in } B_{r/2}(x_i).
\]
By using property \eqref{incatenate}, we can repeat the previous argument for the balls $B_{r/2}(x_{i-1})$ and $B_{r/2}(x_{i+1})$ and so on, up to obtain that $u=0$ almost everywhere on $K$. This clearly contradicts \eqref{ipotesi2}, thus $u>0$ almost everywhere in $K$.
\vskip.2cm\noindent
Let us now assume \eqref{ipotesi}. Since $\Omega$ is connected, there exists a sequence of connected compact sets $K_n\Subset\Omega$ such that 
\[
|\Omega\setminus K_n|<\frac{1}{n}\qquad \mbox{ and }\quad u\not\equiv 0\mbox{ in }K_n.
\]
Then, by the first part of the proof $u>0$ almost everywhere on each $K_n$. 
By letting $n$ go to $\infty$, we get the conclusion.
\end{proof}

\section{Optimal constant for the fractional Hardy inequality}
\label{sec:B}

In Section \ref{sec:6} we used that $u(x)=|x|^{-\beta}$ is a solution of
\[
(-\Delta_p)^s u=C(\beta)\, \frac{u^{p-1}}{|x|^{s\,p}},\qquad \mbox{ in }\mathbb{R}^N\setminus\{0\}.
\] 
In this appendix we discuss some features of the constant 
\[
C(\beta)=2\,\int_{\mathbb{R}^N} \frac{\left||x|^{-\beta}-|y|^{-\beta}\right|^{p-2}\,\left(|x|^{-\beta}-|y|^{-\beta}\right)}{|x-y|^{N+s\,p}}\, dy,\qquad x\in\mathbb{S}^{N-1},
\]
and determines the best constant in \eqref{hardy_frazia}, see equation \eqref{costante} below. Computations are very much the same as in the paper \cite{FS} by Frank and Seiringer, up to some simplifications. For simplicity we focus on the case $N\ge 2$ and $s\,p<N$.
\begin{lm}
\label{lm:puntoperpunto}
Let $N\ge 2$, $0<s<1$ and $1<p<\infty$. For every $0<\varrho<1$ the function
\[
\mathcal{G}(\beta)=\left[1-\varrho^{N-s\,p-\beta\,(p-1)}\right]\,\left[1-\varrho^\beta\right]^{p-1},\qquad \beta>0,
\]
is maximal for $\beta=(N-s\,p)/p$.
\end{lm}
\begin{proof}
We just have to differentiate the function $\mathcal{G}$. Indeed, we have
\[
\begin{split}
\mathcal{G}'(\beta)&=(p-1)\,\log\varrho\, \varrho^{N-s\,p-\beta\,(p-1)}\,\left[1-\varrho^\beta\right]^{p-1}\\
&-(p-1)\,\varrho^\beta\, \log\varrho\,\left[1-\varrho^{N-s\,p-\beta\,(p-1)}\right]\,\left[1-\varrho^\beta\right]^{p-2},
\end{split}
\] 
so that
\[
\mathcal{G}'(\beta)\ge 0\quad \Longleftrightarrow\quad \varrho^{N-s\,p-\beta\,(p-1)}\,\left[1-\varrho^\beta\right]-\varrho^\beta\,\left[1-\varrho^{N-s\,p-\beta\,(p-1)}\right]\le 0
\]
that is if and only if $\beta$ is such that  
\[
\varrho^{N-s\,p-\beta\,(p-1)}\le \varrho^\beta.
\]
By passing to the logarithm, we obtain the assertion.
\end{proof}
We can then determine the maximal values of $C(\beta)$.
\begin{lm}
Let $N\ge 2$, $0<s<1$ and $1<p<\infty$. For every $\beta>0$ we have
\[
0<C(\beta)\le C\left(\frac{N-s\,p}{p}\right).
\]
\end{lm}
\begin{proof}
We recall that
\[
C(\beta)=2\,\int_{\mathbb{R}^N} \frac{\left||x|^{-\beta}-|y|^{-\beta}\right|^{p-2}\, \left(|x|^{-\beta}-|y|^{-\beta}\right)}{|x-y|^{N+s\,p}}\, dy,\qquad \mbox{ for every }x\in\mathbb{S}^{N-1},
\]
and the right-hand side is independent of $x\in\mathbb{S}^{N-1}$.
Thus we have
\[
C(\beta)=\frac{2}{\mathcal{H}^{N-1}(\mathbb{S}^{N-1})}\,\int_{\mathbb{S}^{N-1}}\int_{\mathbb{R}^N} \frac{\left||x|^{-\beta}-|y|^{-\beta}\right|^{p-2}\, \left(|x|^{-\beta}-|y|^{-\beta}\right)}{|x-y|^{N+s\,p}}\, dy\, d\mathcal{H}^{N-1}(x).
\]
We observe that for every $y\in\mathbb{R}^N\setminus\{0\}$ we can write
\[
\mathbb{S}^{N-1}=\bigcup_{t\in[-1,1]} \Sigma_t(y),
\]
where
\[
\Sigma_t=\left\{x\in\mathbb{S}^{N-1}\, :\, \left\langle x,\frac{y}{|y|}\right\rangle=t\right\}\simeq \sqrt{1-t^2}\,\mathbb{S}^{N-2}. 
\]
By using this and exchanging the order of integration, we then obtain
\[
\begin{split}
C(\beta)=\frac{2}{\mathcal{H}^{N-1}(\mathbb{S}^{N-1})}&\,\int_{\mathbb{R}^N}\int_{\mathbb{S}^{N-1}} \frac{\left||x|^{-\beta}-|y|^{-\beta}\right|^{p-2}\, \left(|x|^{-\beta}-|y|^{-\beta}\right)}{|x-y|^{N+s\,p}}\, d\mathcal{H}^{N-1}(x)\,dy\\
&=\frac{\mathcal{H}^{N-2}(\mathbb{S}^{N-2})}{\mathcal{H}^{N-1}(\mathbb{S}^{N-1})}\,\int_{\mathbb{R}^N}\Big[|1-|y|^{-\beta}|^{p-2}\, (1-|y|^{-\beta})\\
&\times\,\int_{-1}^1 \frac{(1-t^2)^\frac{N-3}{2}}{(1-2\,t\,|y|+|y|^2)^\frac{N+s\,p}{2}}\, \,dt\Big]\,dy
\end{split}
\]
where we used that
\[
|x-y|=\sqrt{1-2\, t\,|y|+|y|^2},\qquad \mbox{ for every } x\in \Sigma_t.
\]
We now set
\begin{equation}
\label{fihona}
\Phi(\varrho)=\mathcal{H}^{N-2}(\mathbb{S}^{N-2})\,\int_{-1}^1 \frac{(1-t^2)^\frac{N-3}{2}}{(1-2\,t\,\varrho+\varrho^2)^\frac{N+s\,p}{2}}\, dt,
\end{equation}
then by using polar coordinates we get
\[
\begin{split}
C(\beta)&= 2\,\int_0^\infty \varrho^{N-1}\, \left|1-\varrho^{-\beta}\right|^{p-2}\, (1-\varrho^{-\beta})\, \Phi(\varrho)\,d\varrho\\
&=-2\,\int_0^1 \varrho^{N-1}\,\left|1-\varrho^{-\beta}\right|^{p-1}\, \Phi(\varrho)\,d\varrho\\
&+2\,\int_1^\infty \varrho^{N-1}\,\left|1-\varrho^{-\beta}\right|^{p-1}\, \Phi(\varrho)\,d\varrho.
\end{split}
\]
We now perform the change of variable $\varrho=r^{-1}$ in the second integral and observe that
\[
\Phi(1/r)=r^{N+s\,p}\,\Phi(r).
\]
Thus we obtain
\[
\begin{split}
C(\beta)&=-2\,\int_0^1 \varrho^{N-1}\,\left|1-\varrho^{-\beta}\right|^{p-1}\, \Phi(\varrho)\,d\varrho+2\,\int_0^1 \varrho^{-1+s\,p}\,\left|1-\varrho^{\beta}\right|^{p-1}\, \Phi(\varrho)\,d\varrho\\
&=-2\,\int_0^1 \varrho^{N-1-\beta\,(p-1)}\,\left|\varrho^\beta-1\right|^{p-1}\, \Phi(\varrho)\,d\varrho+2\,\int_0^1 \varrho^{-1+s\,p}\,\left|1-\varrho^{\beta}\right|^{p-1}\, \Phi(\varrho)\,d\varrho\\
&=2\,\int_0^1 \varrho^{s\,p-1}\,\left[1-\varrho^{N-s\,p-\beta\,(p-1)}\right]\,\left|1-\varrho^{\beta}\right|^{p-1}\, \Phi(\varrho)\,d\varrho.
\end{split}
\]
We now observe that the term into square brackets is positive if 
\[
\beta\le \frac{N-s\,p}{p-1}.
\]
Moreover, thanks to Lemma \ref{lm:puntoperpunto} the integrand is maximal for 
\[
\beta=\frac{N-s\,p}{p},
\]
and thus we get the conclusion.
\end{proof}
\begin{oss}
Observe that we have 
\begin{equation}
\label{costante}
C\left(\frac{N-s\,p}{p}\right)=2\,\int_0^1 \varrho^{s\,p-1}\,\left[1-\varrho^\frac{N-s\,p}{p}\right]^{p}\, \Phi(\varrho)\,d\varrho,
\end{equation}
where the function $\Phi$ is defined in \eqref{fihona}. Thus this is the best constant in the Hardy inequality \eqref{hardy_frazia} (see \cite[Section 3.3]{FS} for more details).
\end{oss}

\end{document}